\documentclass[a4paper,11pt]{article}
\usepackage{graphicx}
\usepackage[T1]{fontenc}
\usepackage[utf8]{inputenc}
\usepackage[english]{babel}
\usepackage{amsmath,amsthm,amssymb,amsfonts}
\usepackage{lmodern}
\usepackage{subcaption}
\usepackage{geometry}
\usepackage{framed}
 \usepackage{color}
\definecolor{darkgreen}{rgb}{0.00,0.5,0.00}
\usepackage{hyperref}
\hypersetup{
	colorlinks=true,        % false: boxed links; true: colored links
	linkcolor=blue,         % color of internal links
	citecolor=darkgreen,         % color of links to bibliography
	urlcolor=blue           % color of external links
}

\usepackage{cite}
\usepackage{algorithmic}
\usepackage{algorithm}
\usepackage{enumerate}

\newtheorem{theorem}{Theorem}
\theoremstyle{definition}
\newtheorem{lemma}{Lemma}

\newtheorem{remark}{Remark}

\newcommand{\n}[1]{\left\|#1 \right\|}

\newcommand{\la}{\lambda}

\renewcommand{\t}{\tau}
\newcommand{\s}{\sigma}
\newcommand{\e}{\varepsilon}

\newcommand{\N}{\mathbb N}
\newcommand{\Z}{\mathbb Z}

\newcommand{\x}{\bar x}
\newcommand{\y}{\bar y}
\newcommand{\z}{\bar z}

\newcommand{\lr}[1]{\left\langle #1\right\rangle}
\newcommand{\Hilbert}{\mathcal{H}}
\newcommand{\setto}{\rightrightarrows}
\newcommand{\wto}{\rightharpoonup}

\renewcommand{\empty}{\varnothing}

\DeclareMathOperator{\prox}{prox}

\DeclareMathOperator{\id}{Id}
\DeclareMathOperator{\zer}{zer}
\DeclareMathOperator{\fix}{Fix}

\title{ Shadow Douglas--Rachford Splitting for Monotone Inclusions}
\author{Ern\"o Robert Csetnek\footnote{Faculty of Mathematics,
        University of Vienna, \href{mailto:ernoe.robert.csetnek@univie.ac.at}{ernoe.robert.csetnek@univie.ac.at}} \and Yura Malitsky$^\dagger$
    \and  Matthew K. Tam\footnote{Institute for Numerical and
        Applied Mathematics, University of G\"ottingen,
        \href{mailto:y.malitskyi@math.uni-goettingen.de}{y.malitskyi@math.uni-goettingen.de},
        \href{mailto:m.tam@math.uni-goettingen.de}{m.tam@math.uni-goettingen.de}}}
\date{}

\begin{document}
\maketitle

\begin{abstract}
In this work, we propose a new algorithm for finding a zero in the sum of two monotone operators where one is assumed to be single-valued and Lipschitz continuous. This algorithm naturally arises from a non-standard discretization of a continuous dynamical system associated with the Douglas--Rachford splitting algorithm. More precisely, it is obtained by performing an explicit, rather than implicit, discretization with respect to one of the operators involved. Each iteration of the proposed algorithm requires the evaluation of one forward and one backward operator.
\end{abstract}

    dynamical systems
\medskip

\textbf{MSC2010.} {
49M29, % methods involving duality
90C25, % convex programming
47H05, % monotone operators (with respect to duality)
47J20, % Variational and other types of inequalities involving nonlinear operators
65K15  % Numerical methods for variational inequalities and related
% problems
}
\section{Introduction}\label{intro}

The study of continuous time dynamical systems associated with
iterative algorithms for solving optimization problems has a long
history which can be traced back at least to 1950s
\cite{gavurin,arrow1957}. The relationship between the continuous and
discrete versions of an algorithm provides a unifying perspective
which gives insights into their behavior and properties. As we will
see in this work, this includes suggesting new algorithmic schemes as
well as appropriate Lyapunov functions for analyzing their convergence
properties. The interplay between continuous and discrete dynamical
systems has been studied by many authors including
\cite{polyak1964some,al1968continuous,antipin1994minimization,bolte2003continuous,peypouquet2010evolution,abbas2014newton,boct2017dynamical,attouch2018convergence,banert2015forward}.

The following well-known idea will help to motivate the approach used
in this work.  Let $\Hilbert$ be a real Hilbert space and suppose
${B:\Hilbert\to\Hilbert}$ is a maximal monotone operator. Consider the
monotone equation
   \begin{equation}\label{simple:0}
   \text{find}~x\in\Hilbert\quad \text{such that}\quad 0=B(x),
   \end{equation}
to which the following continuous time dynamical system can be attached
\begin{equation}
\label{simple:1}
\dot{x}(t) = - B(x(t)).
\end{equation}
Let $\la>0$. We now devise two iterative algorithms for solving
\eqref{simple:0} by using different discretizations of $\dot{x}(t)$ in
\eqref{simple:1}. To this end, let us first approximate the trajectory
$x(t)$ in \eqref{simple:1} by discretizing at the points $(k\la)_{k\in
    \Z_+}$, and denote the discretized trajectory by $x_k := x(k\la )$. 

Now, on one hand, using the forward discretization $\dot{x}(t) \approx \frac{x_{k+1}-x_k}{\la}$ gives
\begin{equation}
\label{eq:forw}
x_{k+1} = x_k - \la B(x_k).
\end{equation}
In the particular case when $B$ is the gradient of a function, \eqref{eq:forw} is nothing more than the classical \emph{gradient descent method}. On the other hand, using the backward discretization $\dot{x}(t) \approx \frac{x_{k}-x_{k-1}}{\la}$ gives
\begin{equation}
\label{eq:back}
x_{k} = J_{\la B}(x_{k-1}),
\end{equation}
where $J_{A}:=(\id+A)^{-1}$ denotes the \emph{resolvent} of a (potentially multi-valued) maximal
monotone operator $A:\Hilbert\setto A$. This iteration is precisely
the \emph{proximal point algorithm} for the monotone inclusion
\eqref{simple:0}. It is worth emphasizing that \eqref{eq:forw} and
\eqref{eq:back} are different iterative algorithms which, in general,
do not converge under the same conditions. In particular, if $B$ is monotone but not cocoercive, then \eqref{eq:back} converges to a solution for any $\la>0$ whereas the \eqref{eq:forw} does not. Nevertheless, both algorithms correspond to the same continuous dynamical system \eqref{simple:1}.

In this work, we exploit the same type relationship between continuous and discrete dynamical systems to discover a new algorithm for monotone inclusions of the form
\begin{equation}\label{eq:zero 2sum}
\text{find}~x\in\Hilbert\text{ such that }0\in (A+B)(x),
\end{equation}
where $A:\Hilbert\setto\Hilbert$ and $B:\Hilbert\to\Hilbert$ are
(maximally) monotone operators with $B$ Lipschitz
continuous (but not necessarily cocoercive). More precisely, by using
a non-standard discretization of the continuous time \emph{Douglas--Rachford
algorithm}, we obtain
\begin{equation}
\label{eq:fdr intro}
x_{k+1} = J_{\la A }\bigl(x_k - \la B(x_k)\bigr) - \la \bigl(B(x_k)-
B(x_{k-1})\bigr),
\end{equation}
which, as we will show, converges weakly to a solution of \eqref{eq:zero 2sum} whenever $\la\in(0,\frac{1}{3L})$. Note also that, by choosing the operators $A$ and $B$ appropriately, the setting of \eqref{eq:zero 2sum} covers smooth-nonsmooth convex minimization, monotone inclusions through duality, and saddle point problems with smooth convex-concave couplings. For further details, see~\cite{malitsky2018forward}.

Despite substantial progress in monotone operator theory, there are
not so many original splitting algorithms for solving monotone
inclusions of form \eqref{eq:zero 2sum} which use forward evaluations
of $B$. Tseng's \emph{forward-backward-forward algorithm}
\cite{T2000}, published in 2000, was the first such method capable of
solving \eqref{eq:zero 2sum}. Until recently, this was the only known
method with these properties, however there has been progress in the
area with the discovery of further methods having this property
\cite{johnstone2018projective,johnstone2019single,malitsky2018forward}. In
this connection, see also \cite{combettes2012primal,bello2015variant}.

The remainder of this work is organized as follows. In
Section~\ref{sec:dr}, we discuss the classical Douglas--Rachford and
study an alternative form of its continuous time dynamical system. In
Section~\ref{s:fdr}, we discretize this alternative form to obtain
\eqref{eq:fdr intro} and prove its convergence. In
Section~\ref{sec:pd}, we briefly show how the same idea can be applied
to derive a new primal-dual algorithm. Section~\ref{s:end} concludes our
work by suggesting avenues for further investigation.

\section{From the Discrete to the Continuous}
\label{sec:dr}
The \emph{Douglas--Rachford method} is an algorithm for finding a zero
in the sum of maximally monotone operators, $A$ and $B$. This popular
\emph{splitting method} works by only requiring the evaluation of the
resolvents of each of the operators individually, rather than the
resolvent of  their sum. The method was first formulated for
solving linear equations in \cite{douglas1956numerical} and later
generalized to monotone inclusions in \cite{lions-mercier}.

The method can be compactly described as the fixed point iteration
\begin{equation}\label{dr-1}
  z_{k+1} 
  = \left(\frac{\id + R_{\la A} R_{\la B}}{2}\right)z_k,
\end{equation}
where $R_{\la B}=2J_{\la B}-\id$ denotes the \emph{reflected
    resolvent} of a monotone operator $\la B$. Its behavior is summarized in the following theorem.
\begin{theorem}
  (\cite[Theorem 25.6]{BC2010}).
  Let $A\colon \Hilbert \setto \Hilbert$ and $B\colon \Hilbert \setto
  \Hilbert$ be maximally monotone operators with
  $\zer(A+B) \neq \empty$. Let $\la>0$ and $z_0\in\Hilbert$. Then the sequence $(z_k)$, generated by \eqref{dr-1}, 
  satisfies
  \begin{enumerate}[(i)]
  \item $(z_k)$ converges weakly to a point $z\in
    \fix (R_{\la A} R_{\la B})$.
  \item $(J_{\la B}z_k)$ converges weakly to $J_{\la B}z \in \zer(A+B)$.
  \end{enumerate}
\end{theorem}

The iteration \eqref{dr-1} can be viewed as a discretization of the continuous time dynamical system
\begin{equation}\label{dr-3}
  \dot{z}(t)+z(t)
  = \left(\frac{\id + R_{\la A} R_{\la B}}{2}\right)z(t),
\end{equation}
where the discretization $\dot{z}(t)\approx z_{k+1}-z_k$ and
$z(t)\approx z_k$ are used. Since the operator $R_{\la A} R_{\la B}$
is nonexpansive (\emph{i.e.,} $1$-Lipschitz), the Picard-Lindel\"of theorem
\cite[Theorem 2.2]{granas2013fixed} implies that, for any $z_0\in \Hilbert$, there exists a unique
trajectory $z(t)$ satisfying \eqref{dr-3} and the initial
condition $z(0) = z_0$.

Let us now express this dynamical system in an alternative form. First, by using the definition of the reflected resolvent, we observe that \eqref{dr-3} can be written as
\begin{equation}
  \label{dyn2}
  	\dot{z}(t) = J_{\la A}\bigl(2J_{\la B}(z(t)) - z(t) \bigr) - J_{\la B}(z(t)).
\end{equation}
Denote $x(t) = J_{\la B}(z(t))$ and $y(t) = z(t) - x(t)$. Clearly, $y(t)\in \la B(x(t))$. Then we have 
\begin{equation}
  \label{z(t)}
  z(t) = x(t) + y(t),\quad \text{and} \quad \dot{z}(t) = \dot{x}(t) + \dot{y}(t).
\end{equation}
By using these identities to eliminate $z$ from \eqref{dyn2}, we obtain
  \begin{equation}\label{dyn3}
	\begin{aligned}
	\dot{x}(t)+x(t) &= J_{\la A}\left( x(t)- y(t)\right) -\dot{y}(t), \\
	y(t)            &\in \la B(x(t)).
      \end{aligned}
    \end{equation}
This system can be viewed as the continuous dynamical system associated with the shadow trajectories, $x(t)$, of the Douglas--Rachford system \eqref{dr-3} specified by $z(t)$. In particular, this fact implies the existence of the trajectories $x(t)$ and $y(t)$. In a later section, we will use a discretization of this system to obtain a new splitting algorithm. 

We begin with a theorem concerning the asymptotic behavior of \eqref{dyn3}. Although this result can be obtained, with some work, from \cite[Theorem~6]{boct2017dynamical}, we give a more direct proof which serves the additional purpose of providing insights useful for the analysis of the discrete case. We require the following two preparatory lemmas.

\begin{lemma}\label{lem:demiclosed}
Let $\la>0$. Suppose $A\colon \Hilbert \setto \Hilbert$ and $B\colon \Hilbert \setto  \Hilbert$ are maximally monotone operators. Then the set-valued operator on $\Hilbert\times\Hilbert$ defined by
    \begin{equation}\label{eq:mono sum}
    \binom{x}{y}\mapsto \left(\begin{bmatrix}\la A\\(\la B)^{-1}\end{bmatrix}+\begin{bmatrix}
    0 & \id \\ -\id & 0 \\
    \end{bmatrix}\right)\binom{x}{y},
    \end{equation} 
is demiclosed. That is, its graph is a sequentially closed set in the weak-strong topology.
\end{lemma}	
\begin{proof}
Note that the operator in \eqref{eq:mono sum} is maximally monotone as the sum of two maximally monotone, the latter having full domain \cite[Corollary~24.4(i)]{BC2010}. Since maximally monotone operators are demiclosed \cite[Proposition~20.32]{BC2010}, the result follows.
\end{proof}	
 
Although the following lemma is a direct consequence of~\cite[Lemma~5.2]{abbas2014newton}, we include its explicit statement for the convenience of the reader.
\begin{lemma}\label{lem:l2}
Suppose $T\colon\Hilbert\to\Hilbert$ is $L$-Lipschitz continuous. If $\dot{z}(t)=T(z(t))$ and $\int_0^\infty\n{\dot{z}(t)}^2\,dt<+\infty$, then $\dot{z}(t)\to 0$ as $t\to+\infty$.
\end{lemma}	
\begin{proof}		
Since $T$ is $L$-Lipschitz continuous, \cite[Remark~1]{boct2017dynamical} implies that $\ddot{z}$ exists almost everywhere and that
  $ \n{\ddot{z}(t)} = \n{\frac{d}{dt}Tz(t)} \leq L\n{\dot{z}(t)} $
for almost all $t\geq 0$. From this it follows that $\int_0^\infty\n{\ddot{z}(t)}^2\,dt<+\infty$. We also have
  $$ \frac{d}{dt}\n{\dot{z}(t)}^2 = 2\lr{\ddot{z}(t),\dot{z}(t)} \leq \n{\ddot{z}(t)}^2 + \n{\dot{z}(t)}^2.$$
Since the right hand side is integrable, \cite[Lemma~5.2]{abbas2014newton} yields the result.
\end{proof}

The following theorem is our main result regarding the asymptotic behavior of \eqref{dyn3}.
\begin{theorem}\label{th:main-1}
  Let $A\colon \Hilbert \setto \Hilbert$ and $B\colon \Hilbert \setto
  \Hilbert$ be maximally monotone operators with $\zer(A+B) \neq
  \empty$. Let $\lambda>0$ and $x_0\in \Hilbert$. Then the
  trajectories $x(t)$, $y(t)$, generated by \eqref{dyn3} with initial condition $x(0)=x_0$, satisfy
  \begin{enumerate}[(i)]
  \item $x(t)$ converges weakly to a point $\x\in \zer(A+B)$.
  \item $y(t)$ converges weakly to a point $\y\in \la B(\x)\cap(- \la A(\x))$.
  \end{enumerate}
\end{theorem}
\begin{proof}		
	Let $\x\in\zer(A+B)$ and $\y\in\la B(\x)\cap(-\la A(\x))$. Denote $\z=\x+\y$ and $z(t)=x(t)+y(t)$. By using monotonicity of $\la A$ followed by monotonicity of $\la B$, we obtain
	\begin{equation}\label{eq:key cont}
	\begin{aligned}
	0  &\leq \lr{\dot{x}(t)+ y(t) +\dot{y}(t)-\y,\x-\dot{x}(t)-x(t) - \dot{y}(t)} \\
	&= \lr{\dot{z}(t),\z-z(t)}  -\n{\dot{z}(t)}^2 - \lr{y(t)-\y,x(t)-\x} \\
	&\leq  -\frac{1}{2}\frac{d}{dt}\n{z(t)-\z}^2 - \n{\dot{z}(t)}^2.
	\end{aligned}
	\end{equation}
	In particular, this shows that $\n{z(t)-\z}^2$ is decreasing,
        hence $\lim_{t\to\infty}\n{z(t)-\z}$ exists, and that
        $\int_0^\infty\n{\dot{z}(t)}^2\, dt<+\infty$. The latter combined with Lemma~\ref{lem:l2} implies that $\dot{z}(t)\to0$ as $t\to\infty$. Monotonicity of $\la B$ then yields
	\begin{align*}
	\n{z(t)-\z}^2 
	&= \n{x(t)-\x}^2 + 2\lr{x(t)-\x,y(t)-\y} + \n{y(t)-\y}^2 \\
	&\geq \n{x(t)-\x}^2 + \n{y(t)-\y}^2,
	\end{align*}  
	from which it follows that $x(t)$ is bounded. By using the definition of the resolvent $J_{\la A}$, we can express \eqref{dyn3} in the form
    \begin{equation}
    \label{eq:lift cont}
    -\binom{ \dot{z}(t) }{ \dot{z}(t) } \in \left(\begin{bmatrix}\la A\\(\la B)^{-1}\end{bmatrix}+\begin{bmatrix}
    0 & \id \\ -\id & 0 \\
    \end{bmatrix}\right)\binom{\dot{z}(t)+x(t)}{z(t)-x(t)}.
    \end{equation} 
    Let $(x,z)$ be a weak sequential cluster point of the bounded trajectory $(x(t),z(t))$. Taking the limit along this subsequence in \eqref{eq:lift cont}, using Lemma~\ref{lem:demiclosed}, and unraveling the resulting expression gives
    \begin{equation}\label{eq:key limit cont}
    \left\{\begin{aligned}
    0 &\in \la A(x)+(z-x) \\
    x &\in (\la B)^{-1}(z-x)\\
    \end{aligned}\right. 
    \quad\implies\quad
    \left\{\begin{aligned}
    x &\in \zer(A+B)\\
    z &\in x+\la B(x)\\
    \end{aligned}\right. 
    \end{equation}
    In particular, by combining \eqref{eq:key cont} with \eqref{eq:key limit cont}, we deduce that $\lim_{t\to+\infty}\n{z(t)-z}^2$ exists. Applying Opial's lemma \cite[Lemma~4]{boct2017dynamical} then shows that $z(t)$ converges weakly to a point $\z\in\x+\la B(\x)$ where $\x$ is a weak sequential cluster point of $x(t)$. 
    The definition of $J_{\la B}$ then yields $\x=J_{\la B}(\z)$, which implies that $J_{\la B}(\z)$ is the unique cluster point of $x(t)$. The trajectory $x(t)$ therefore converges weakly to a point $\x\in\zer(A+B)$. To complete the proof, simply note that $y(t)=z(t)-x(t)\wto \z-\x \in \la B(\x)\cap(-\la A(\x))$ as $t\to+\infty$.
\end{proof}

\section{From the Continuous to the Discrete}
\label{s:fdr}
In this section, we devise a new splitting algorithm by considering
different discretizations of the dynamical system \eqref{dyn3}. For the remainder of this work, we will suppose that $B$ is a single-valued operator. In this case, the system~\eqref{dyn3} simplifies to
\begin{equation}\label{dyn3-v2}
\begin{aligned}
\dot{x}(t)+x(t) &= J_{\la A}\left( x(t)- y(t)\right) -\dot{y}(t), \\
y(t) &= \la B(x(t)).
\end{aligned}
\end{equation}
In order to discretize this system, let us replace $x(t)\approx x_k$ and $y(t)\approx y_k$. As two derivatives appear in \eqref{dyn3-v2}, there are many combinations of possible discretizations. One involves using forward discretizations of both $\dot{x}(t)$ and $\dot{y}(t)$, that is,
\begin{equation}\label{dr:2forw}
\dot{x}(t) \approx x_{k+1}-x_k,\quad \dot{y}(t) \approx y_{k+1}-y_k.
\end{equation}
Under this discretization, \eqref{dyn3-v2} becomes
\begin{equation}
\label{dr normal shadows}
 x_{k+1} = J_{\la A}(x_k-\la B(x_k)) - \la \bigl(B(x_{k+1})-B(x_k)\bigr).
\end{equation}
As written, this expression does not given rise to a useful algorithm, since $x_{k+1}$ appears on both sides of the equation. However, we note that by taking $z_k=x_k+y_k=(I+\la B)x_k$ and rearranging, we obtain
\begin{equation*}
z_{k+1} = z_k + J_{\la A}(2J_{\la B}z_k-z_k) - J_{\la B}(z_k),
\end{equation*}
which is precisely the usual Douglas--Rachford
algorithm given in \eqref{dr-1}.

To derive a new algorithm, we consider a different discretization of
\eqref{dyn3-v2}. To this end, we perform a forward discretization of
$\dot{x}(t)$ and a backward discretization of $\dot{y}(t)$, that is,
\begin{equation}\label{dr:1back}
\dot{x}(t) \approx x_{k+1}-x_k,\quad \dot{y}(t) \approx y_{k}-y_{k-1},
\end{equation}
Under this discretization, \eqref{dyn3} becomes
\begin{equation}
\label{fbr}
x_{k+1} = J_{\la A }\bigl(x_k - \la B(x_k)\bigr) - \la \bigl(B(x_k)-
B(x_{k-1})\bigr).
\end{equation}
Although not surprising, it is interesting to note that \eqref{dr normal shadows} and \eqref{fbr} only differ in the indices which appear in the last two terms. In particular, in this expression, $x_{k+1}$ does not appear on the right-hand side.

Before turning our attention to the convergence properties of this iteration, we make the following remark.

\begin{remark}
Backward/forward discretizations of a derivative usually correspond to the same type of step in their discrete counterpart of the algorithms. This is, for instance, the case for the forward-backward method which includes the discussion from Section~\ref{simple:1} as a special case. It is curious to note, however, that forward (resp.\ backward) discretization gave rise to backward  (resp.\ forward) operators in the discrete counterparts. In particular, two forward discretizations of \eqref{dyn3-v2} gave rise the Douglas--Rachford algorithm which has two backward steps whereas one forward and one backward discretization produced a method also having one forward and one backward step. 
\end{remark}	

We now prove the following preparatory lemma, which might be interesting in its
own right due to the very general form of the recurrent relation.
\begin{lemma}\label{lemma:abs}
    Let $A\colon \Hilbert \setto \Hilbert$ be a maximal monotone
    operator and let $(y_k)\subset \Hilbert$ be an arbitrary sequence. Let $x_0\in\Hilbert$ and consider $(x_k)$ defined by
    \begin{equation}
        \label{abstr_seq}
             x_{k+1} = J_{A}(x_k-y_k) - (y_k - y_{k-1}),\quad \forall k\in\mathbb{N}.
     \end{equation}
     Then, for all $x\in \Hilbert$ and $y \in -A(x)$, we have
\begin{align}
  \label{eq:disc1_abs}
  \|(x_{k+1} & +y_k)  - (x+y)\|^2  \leq
               \n{(x_{k}+y_{k-1}) - (x+y)}^2 - 2 (y_k - y, x_k - x)
               \nonumber \\
             & + 4\lr{y_k-y_{k-1}, x_k-x_{k+1}} -\n{x_{k+1}-x_k}^2
               -3\n{y_k-y_{k-1}}^2.
\end{align}
\end{lemma}
\begin{proof}
By the definition of the resolvent and \eqref{abstr_seq}, it follows that
\begin{equation}\label{eq:mono}
	x_{k+1}-x_k + y_k + ( y_k - y_{k-1}) \in - A\bigl(x_{k+1} + y_k-y_{k-1}\bigr).
\end{equation}
Since $-y \in A(x)$ and $A$ is monotone, we have
\begin{equation*}
	0\leq \lr{x_{k+1}-x_k + y_k +
		(y_k-y_{k-1}) - y, x - x_{k+1}- (y_k-y_{k-1}) },
\end{equation*}
which is equivalent to
\begin{multline}\label{eq:new_terms}
	0\leq \lr{x_{k+1}-x_k,x-x_{k+1}} + \lr{y_k-y,x-x_{k+1}} + \lr{y_k-y_{k-1},x-x_{k+1}} \\+ \lr{y_k -
		y_{k-1},x_k - x_{k+1}} + \lr{y_k-y_{k-1}, y- y_k}-
            \n{y_k - y_{k-1}}^2 .
\end{multline}
To simplify \eqref{eq:new_terms}, we note that
\begin{align*}
  2\lr{x_{k+1}-x_k, x-x_{k+1}} & = \n{x_{k}-x}^2 -
                                 \n{x_{k+1} -x_k}^2 - \n{x_{k+1}
                                 -x}^2,\\ 
  2 \lr{y_k-y_{k-1}, y-y_k } & = \n{y_{k-1} -y}^2 -\n{y_k-y_{k-1}}^2 
                               -\n{y_k -y}^2 ,\\
  \lr{y_k-y_{k-1}, x-x_{k+1}} & =\lr{y_k-y_{k-1},x_k-x_{k+1}}\\
                               &  \phantom{\qquad} +\lr{y_{k-1}-y, x_k-x}   +\lr{y-y_k, x_k-x}.
\end{align*}
Now, using the above three identities in \eqref{eq:new_terms}, we
obtain
\begin{multline}\label{eq:last_in_lemma}
    \n{x_{k+1}-x}^2 + 2\lr{y_k-y, x_{k+1}-x} + \n{y_k-y}^2 \leq
    \n{x_{k}-x}^2 + 2\lr{y_{k-1}-y, x_{k}-x} +  \n{y_{k-1}-y}^2 \\
    + 4\lr{y_k-y_{k-1}, x_k-x_{k+1}} -\n{x_{k+1}-x_k}^2 -
    3\n{y_k-y_{k-1}}^2 - 2\lr{y_k-y, x_k-x}.
\end{multline}

The equivalence between the last inequality and \eqref{eq:disc1_abs}
is now obvious.
\end{proof}

Since \eqref{dr:1back} is of the form specified by
Lemma~\ref{lemma:abs}, this lemma suggests one possible way to prove
convergence of \eqref{dr:1back}: the quantity
$\n{x_k+y_{k-1} - x - y}^2$ will be decreasing if the other terms in
the right hand-side of \eqref{eq:disc1_abs} can be estimated
appropriately. The following theorem, which is our main result
regarding convergence of \eqref{fbr}, makes use of this observation.
\begin{theorem}\label{th:main-2}
    Let $A:\Hilbert\setto\Hilbert$ be maximally monotone and
    $B:\Hilbert\to\Hilbert$ be monotone and $L$-Lipschitz with
    $\zer(A+B)\neq\empty$. Let $\e>0$,
    $\la \in \left[\e,\frac{1-3\e}{3L}\right]$ and let
    $x_0,x_{-1}\in\Hilbert$. Then the sequence $(x_k)$, generated by
    \eqref{fbr}, satisfies
        \begin{enumerate}[(i)]
            \item $(x_k)$ converges weakly to a point
            $\overline{x}\in \zer(A+B)$.
            \item $(B(x_k))$ converges weakly to $B(\bar{x})$.
        \end{enumerate}
\end{theorem}
\begin{proof}
    Let $x\in\zer(A+B)$ and set $y=\la B(x)\in -\la A(x)$. Since
    \eqref{fbr} of the form specified by \eqref{abstr_seq}, we apply
    Lemma~\ref{lemma:abs} to the monotone operator $\la A$ with
    $y_k=\la B(x_k)$ to deduce that the inequality
    \eqref{eq:disc1_abs} holds. Now, using that $B$ is monotone, we
    have $\lr{y_k-y, x_k-x}\geq 0$ and hence
\begin{align}\label{eq:disc1} 
  \|(x_{k+1} & +y_k)  - (x+y)\|^2  \leq
               \n{(x_{k}+y_{k-1}) - (x+y)}^2 
               \nonumber \\
             & + 4\lr{y_k-y_{k-1}, x_k-x_{k+1}} -\n{x_{k+1}-x_k}^2
               -3\n{y_k-y_{k-1}}^2.
\end{align}
Next, we estimate the inner-product in the last line of \eqref{eq:disc1}. To this end, note that Young's inequality gives
\begin{equation}\label{eq:est1}
	2 \lr{y_k-y_{k-1},x_k-x_{k+1}}\leq \frac 1 3 \n{x_{k+1}-x_k}^2 + 3 \n{y_k-y_{k-1}}^2,
\end{equation}
and that Lipschitzness of $B$ yields
\begin{align}\label{eq:est2}
	2 \lr{y_k-y_{k-1},x_k-x_{k+1}}
	& \leq \la L \bigl(\n{x_{k}-x_{k-1}}^2 + \n{x_{k+1}-x_k}^2\bigr).
\end{align}
Combing these two estimates with \eqref{eq:disc1} gives the inequality
      \begin{align*}
        \n{x_{k+1}+y_k -x-y}^2 & + \left(\frac 2 3 -\la L\right)\n{x_{k+1}-x_k}^2\\ &\leq 
        \n{x_{k}+ y_{k-1}-x -y}^2 + \la L \n{x_k-x_{k-1}}^2.
      \end{align*}
By denoting $z_{k}=x_k+y_{k-1}$ and $z=x+y$, the previous inequality implies
      \begin{equation}\label{eq:key}
      	\n{z_{k+1} -z}^2  + \left(\frac{1}{3}+\e\right)\n{x_{k+1}-x_k}^2\leq 
      	\n{z_k - z}^2 + \frac{1}{3} \n{x_k-x_{k-1}}^2,
      \end{equation} 
which telescopes to yield 
\begin{equation*}
  \n{z_{k+1} -z}^2  + \frac{1}{3}\n{x_{k+1}-x_k}^2 +\e\sum_{i=1}^k\n{x_{i+1}-x_i}^2 
  \leq \n{z_1 - z}^2 + \frac{1}{3} \n{x_1-x_{0}}^2.
 \end{equation*}    
From this, it follows that $(z_k)$ is bounded and that
${\n{x_k-x_{k-1}}\to 0}$. The latter, together with Lipschitz continuity of $B$, implies  ${\n{y_{k}-y_{k-1}}\to 0}$ and, consequently, we also have that $\n{z_k-z_{k-1}}\to 0$. Since $z_k = (\id+\la B)x_k + (y_{k-1}-y_k)$, we have
  $$ x_k = J_{\la B}\left( z_k  - (y_{k-1}-y_k) \right).$$
Since $(z_k)$ is bounded, $\n{y_k-y_{k-1}}\to 0$ and $J_{\la B}$ is nonexpansive, it then follows that the sequence $(x_k)$ is also bounded. Also, due to \eqref{eq:key}, we see that the following limit exits $$\lim_{k\to\infty}\left(\n{z_k-z}^2+\frac{1}{3}\n{x_{k+1}-x_k}^2\right)=\lim_{k\to\infty}\n{z_k-z}^2.$$
Now, by using the definition of the resolvent $J_{\la A}$, we can express \eqref{eq:mono} in the form
\begin{equation}
\label{eq:lift inclusion}
-\binom{ z_{k+1}-z_k }{ z_{k+1}-z_k } \in \left(\begin{bmatrix}\la A\\(\la B)^{-1}\end{bmatrix}+\begin{bmatrix}
           0 & \id \\ -\id & 0 \\
         \end{bmatrix}\right)\binom{z_{k+1}-z_k+x_k}{z_{k+1}-x_{k+1}}.
\end{equation} 
Let $(x,z)$ be a weak cluster point of the bounded sequence $(x_k,z_k)$. Taking the limit along this subsequence in \eqref{eq:lift inclusion},  using Lemma~\ref{lem:demiclosed}, and unravelling the resulting expression gives
    \begin{equation}\label{limit}
    \left\{\begin{aligned}
    0 &\in \la A(x)+(z-x) \\
    x &\in (\la B)^{-1}(z-x)\\
    \end{aligned}\right. 
    \quad\implies\quad
    \left\{\begin{aligned}
    x &\in \zer(A+B)\\
    z &\in x+\la B(x)\\
    \end{aligned}\right. 
    \end{equation}
Applying Opial's Lemma \cite[Lemma~2.39]{BC2010} then follows that $(z_k)$ converges weakly to a point $\bar{z}=\x+\la B(\x)$ where $\x$ is weak cluster point of $(x_k)$. But then the definition of $J_{\la B}$ yields that $\x=J_{\la B}(\z)$ which implies that $J_{\la B}(\z)$ is the unique cluster point of $(x_k)$. The sequence $(x_k)$ therefore converges weakly to a point $\x\in\zer(A+B)$. To complete the proof, simply note that
 $y_{k-1}=z_{k}-x_k\wto \bar{z}-\bar{x} = \la B(\bar{x})$ as $k\to\infty$.
\end{proof}

Some remarks regarding Theorem~\ref{th:main-2} and its proof are in order.

\begin{remark}[Continuous and discrete proofs]
    The sequence $z_{k}=x_k+y_{k-1}$ plays a similar role in
    Theorem~\ref{th:main-2} to the trajectory $z(t)=x(t)+y(t)$ in
    Theorem~\ref{th:main-1}. This does however highlight a subtle
    difference between the two proofs~---~in the discrete case, we
    have $x_k=J_{\la B}\bigl(z_k+(y_k-y_{k-1})\bigr)$ whereas, in the
    continuous case, we have $x(t)=J_{\la B}(z(t))$. Note also that
    although our combination of the estimates \eqref{eq:est1} and
    \eqref{eq:est2} for $\lr{y_k-y_{k-1},x_k-x_{k+1}}$ may appear
    somewhat arbitrary, the combination of these two inequalities is
    in fact optimal.
\end{remark}    

\begin{remark}\label{r:stepsize}
Although we were unable to prove so in Theorem~\ref{th:main-2}, we conjecture that the interval in which $\la$ lies can be extended to $\la\in(0,\frac{1}{2L})$.
Our original motivation for considering the continuous dynamical system \eqref{dyn3} did not arise from its connection to the Douglas--Rachford algorithm, but rather it from its connection to the operator splitting method studied in~\cite{malitsky2018forward} given by
\begin{equation}
    \label{siopt}
        x_{k+1} = J_{\la A }\bigl(x_k - \la B(x_k)- \la (B(x_k)-
B(x_{k-1})) \bigr).
\end{equation}
Note that the iterations \eqref{fbr} and \eqref{siopt} look very
similar and, in fact, coincide if $J_A$ is the identity operator. For \eqref{siopt}, convergence has been established when $\la < \frac{1}{2L}$, which is slightly better than for \eqref{fbr}. Thus, in the case that $A=0$, this provides some evidence for the conjecture.

On the other hand, the analysis of dynamical systems corresponding to \eqref{siopt} is more complicated. In particular, a natural candidate for a continuous analogue of \eqref{siopt} is given by
  \begin{equation}\label{dyn4}
	\begin{aligned}
	\dot{x}(t)+x(t) &= J_{\la A}\left( x(t)- y(t) -\dot{y}(t)\right). \\
	y(t) &= \la B(x(t)).
      \end{aligned}
    \end{equation}
Because we are unable to couple the derivatives $\dot{x}(t)$ and $\dot{y}(t)$ in \eqref{dyn4} in general, it is not clear how to prove existence of its trajectory $x(t)$. 
\end{remark}

\section{Primal-Dual Algorithms}\label{sec:pd}
In this section, we use Lemma~\ref{lemma:abs} from Section~\ref{s:fdr} to analyse a new primal-dual algorithm. Consider the bilinear convex-concave saddle point problem
\begin{equation}
    \label{saddle}
    \min_{u\in \Hilbert_1} \max_{v\in \Hilbert_2} \, g(u) + \lr{Ku,v} - f^*(v),
\end{equation}
where $g\colon \Hilbert_1\to (-\infty, +\infty]$,
$f\colon \Hilbert_2 \to (-\infty, +\infty]$ are proper convex lsc
functions, $K\colon \Hilbert_1 \to \Hilbert_2$ is a bounded linear
operator with norm $\n{K}$, and $f^*$ denotes the Fenchel conjugate of $f$.
A popular method to solve this problem is the \emph{primal-dual
method}~\cite{chambolle2011first} defined by
\begin{equation}\label{pdhg}
\begin{aligned}
  u_{k+1} & = \prox_{\t g}(u_k - \t K^*v_k)\\
  v_{k+1} &= \prox_{\s f^*}(v_k + \s K(2u_{k+1}-u_k)).
\end{aligned}
\end{equation}
Under the assumption that the solution set of \eqref{saddle} is non-empty
and that $\t \s \n{K}^2 < 1$, one can prove that the sequence $(u_k, v_k)$ weakly
converges to a saddle point of \eqref{saddle}.

In spirit of \eqref{fbr}, we propose the following novel primal-dual algorithm:
\begin{equation}
    \label{new_pd}
\begin{aligned}
  u_{k+1} & = \prox_{\t g}(u_k - \t K^*v_k)\\
  v_{k+1} &= \prox_{\s f^*}(v_k + \s Ku_{k+1}) + \s (Ku_{k+1}-Ku_k).
\end{aligned}
\end{equation}
In the following theorem, we prove convergence of this algorithm. As one can see,
the conditions required for its convergence are exactly the same as for \eqref{pdhg}.
Rather than present the full proof, we will only focus on the
most important ingredient~--- the fact that $(u_k)$, $(v_k)$ remain
bounded. One this is established, the rest of the proof follows the standard argument, as in Theorem~\ref{th:main-2}.
\begin{theorem}
    Let $g\colon \Hilbert_1 \to (-\infty, +\infty]$,
    $f\colon \Hilbert_2\to (\infty, +\infty]$ be proper convex lsc
    functions and $K\colon \Hilbert_1 \to \Hilbert_2$ be a bounded
    linear operator with norm $\n{K}$ such that the solution set of
    \eqref{saddle} is nonempty. Let $\t \s \n{K}^2 < 1$, let
    $u_0\in \Hilbert_1$, and let $v_0\in \Hilbert_2$. Then the sequence $(u_k, v_k)$,
    generated by \eqref{new_pd}, converges  weakly to a solution of
    \eqref{saddle}.
\end{theorem}
\begin{proof}
    Let $(u, v)$ be a saddle point of \eqref{saddle}. Then the first-order
    optimality conditions give $ -K^*v \in \partial g(u)$
    and $Ku \in \partial f^*(v)$. By applying Lemma~\ref{lemma:abs} for a fixed $k\in \N$ with
    \[A = \t \partial g,\ x_k = u_k,\ y_k = y_{k-1} = \t K^*v_k,\ x=u,\
        y = \t K^*v, \]
    we obtain
    \begin{equation}
  \label{eq:pd-1}
      \n{u_{k+1}  - u}^2  + 2  \t\lr{K^* v_k - K^*v, u_{k+1}-u} \leq     \|u_{k} -u \|^2 -\n{u_{k+1}-u_k}^2,
  \end{equation}
  where, instead of \eqref{eq:disc1_abs}, we used its equivalent
  form~\eqref{eq:last_in_lemma}. Similarly, by applying
  Lemma~\ref{lemma:abs}  for a fixed $k\in \N$ with
    \[A = \s \partial f^*,\ x_k = v_k,\ y_k = -\s Ku_{k+1},\  y_{k-1}
        = -\s Ku_k, \ x= v, \ y = -\s Ku,  \]
    we obtain
    \begin{multline}
        \label{eq:pd-2}
        \n{(v_{k+1}- \s Ku_{k+1}) - (v - \s Ku)}^2  \leq
        \n{(v_{k}- \s Ku_{k}) - (v - \s Ku)}^2 + 2\s \lr{K(u_{k+1}-u),v_k-v}\\
        -\n{v_{k+1}-v_k}^2 - 3\s^2 \n{K(u_{k+1}-u_k)}^2  -
        4\s \lr{Ku_{k+1} - Ku_k, v_k-v_{k+1}} .
    \end{multline}
    
By applying Young's inequality and using the inequality $\t \s \n{K}^2 <1$, we have
\begin{multline}\label{pd:est}
    -4 \lr{Ku_{k+1} - Ku_k, v_k-v_{k+1}} \leq 4\s \n{Ku_{k+1}-Ku_k}^2+\frac{1}{\s}\n{v_{k+1}-v_k}^2 
    \\ \leq \frac{1}{\t}\n{u_{k+1}-u_k}^2 +\frac{1}{\s}\n{v_{k+1}-v_k}^2 +
    3\s \n{K(u_{k+1}-u_k)}^2.
\end{multline}
Now, multiplying \eqref{eq:pd-1} by $1/\t$, \eqref{eq:pd-2} by
$1/\s$, summing these two inequalities, and then using the estimate \eqref{pd:est} yields
\begin{multline}
    \label{pd:final}
    \frac 1 \t \n{u_{k+1} - u}^2 + \frac 1 \s \n{(v_{k+1}- \s
        Ku_{k+1}) - (v - \s Ku)}^2 \\ \leq \frac 1 \t \n{u_{k} -
        u}^2 + \frac 1 \s \n{(v_{k}- \s Ku_{k}) - (v - \s Ku)}^2.
\end{multline} 
By telescoping this inequality, one obtains boundedness of $(u_k)$ and
$(v_k)$.  In fact, a slightly tighter estimation in \eqref{pd:est}
would yield $\n{u_k-u_{k-1}}\to 0$ and $\n{v_k-v_{k-1}}\to 0$ (since
the inequality $\t \s \n{K}^2 < 1$ is strict).
\end{proof}
Although we do not know yet if the proposed scheme \eqref{new_pd} has any
benefits as compared to \eqref{pdhg}, we believe that both algorithms
will perform very similarly. Nevertheless, the fact that the Lyapunov function associated with the analysis of \eqref{new_pd} is different to the one use for \eqref{pdhg} might be of interest for deriving new extensions.

\section{Concluding Remarks/Future Directions}\label{s:end}
In this work, we proposed and analyzed a new algorithm for finding a
zero in the sum of two monotone operators, one of which is assumed to
be Lipschitz continuous. This algorithm naturally arise from a
non-standard  discretization of a continuous dynamical system with the
Douglas--Rachford algorithm.  To conclude, we outline possible
directions for future work.
\begin{itemize}
    \item \textbf{Extending the stepsize:} In our main result,
    Theorem~\ref{th:main-2}, we established convergence whenever
    $\la<\frac{1}{3L}$. However, for the reasons discussed in
    Remark~\ref{r:stepsize}, the upper-bound can be improved to
    $\la <\frac{1}{2L}$, at least when $A=0$. It would be interesting
    to either improve or show, by means of a counterexample, that the
    condition $\la <\frac{1}{3L}$ is optimal. Furthermore, it would
    also be interesting to investigate the optimal convergence rate
    under some additional assumptions, as it was done
    in~\cite{ryu2018operator} for the classical Douglas--Rachford
    algorithm.  
   	
   	\item \textbf{Linesearch:} It would be interesting to
        incorporate a linesearch procedure in the shadow 
        Douglas--Rachford method. Similarly, it makes sense  to
        consider a continuous dynamic scheme with variable steps, as
        it was done, for example, in~\cite{banert2015forward} for
        Tseng's method.
   	
        \item \textbf{Inertial terms:} It is important to study the
        extensions of \eqref{dyn3} and \eqref{fbr}, which incorporate
        additional inertial and relaxed terms, as it was done in the
        recent work~\cite{attouch2018convergence} for the
        forward-backward method. Combining inertial and relaxing
        effects allows one to go beyond the standard bound of $\frac 1 3$
        for the stepsize associated with the inertial term.
        
        \item \textbf{Role of reflection:} Perhaps the most
        interesting and challenging direction for future work is to understand
        why the inclusion of a ``reflection term'' in an algorithm allows for 
        convergence to proven under milder hypotheses. For instance, applied to the saddle
        point problem~\eqref{saddle}, the famous \emph{Arrow--Hurwicz
        algorithm}~\cite{arrow1957} can fail to converge. In contrast, both \eqref{pdhg} and \eqref{new_pd}, which can be viewed its ``reflected'' modifications, do converge. Similarly, for the monotone variational inequality $0\in N_C(x) + B(x)$,
        where $C$ is a closed convex set and $N_C$ is its normal cone,
        the projected gradient algorithm
        \[x_{k+1} = P_C (x_k - \la B(x_k))\] does not work, but its
        ``reflected'' modification~\cite{M2015} given by
        \[x_{k+1} = P_C (x_k - \la B(2x_k-x_{k-1}))\] 
        does converge to a solution. For the more general monotone inclusion ${0 \in A(x) + B(x)}$, the forward-backward method also does not work, however both of its ``reflected''
        modifications, \eqref{fbr} and \eqref{siopt}, do. We note however that although all of 
        aforementioned algorithms share the same
        ``reflected term'', their analyses are not the same. It
         would be interesting to understand deeper reasons for their success.
\end{itemize}

\small
\paragraph{Acknowledgements.}
E.R.~Csetnek was supported by Austrian Science Fund Project P
29809-N32.  Y.~Maltsky was supported by German Research Foundation
grant SFB755-A4.  The authors also would like to thank the Erwin
Sch\"rodinger Institute for their support and hospitality during the
thematic program ``Modern Maximal Monotone Operator Theory: From
Nonsmooth Optimization to Differential Inclusions''.

\bibliographystyle{acm}
\bibliography{biblio}
\end{document}